\documentclass[11pt, letterpaper]{amsart}
\usepackage{hyperref, color}

\begin{document}

%\author{	D\"orte Kreher\footnote{,$^\dagger$ Institute of Mathematics \& Department of Banking and Finance, University of Z\"urich, Z\"urich, Switzerland,  \texttt{doerte.kreher@math.uzh.ch}, \texttt{ashkan.nikeghbali@math.uzh.ch} } 	
%\and Ashkan Nikeghbali$^\dagger$ }
    
  %For {amsart}-class:

   \author{D\"orte Kreher}
   \address{Universit\"at Z\"urich, Z\"urich, Switzerland}
   \email{doerte.kreher@math.uzh.ch}
   
   \author{Ashkan Nikeghbali}
   \address{Universit\"at Z\"urich, Z\"urich, Switzerland}
   \email{ashkan.nikeghbali@math.uzh.ch}
  
	\begin{abstract}
	In this note we introduce a new kind of augmentation of filtrations along a sequence of stopping times. This augmentation is suitable for the construction of new probability measures associated to a positive strict local martingale as done in \cite{split1}, while it is on the other hand rich enough to make classical results from stochastic analysis hold true on some stochastic interval of interest. 
	\end{abstract}
	
  %\subjclass{}
  %\keywords{Strict local martingales, augmentation of filtration, usual assumptions}

  %\thanks{This research was partially supported by .} 

   \title{A new kind of augmentation of filtrations suitable for a change of probability measure by a strict local martingale}
   \date{\today}
   
\newcommand{\R}{\mathbb R}
\newcommand{\Z}{\mathbb Z}
\newcommand{\E}{\mathbb E}
\newcommand{\1}{\mathbbm I}
\newcommand{\indic}{\mathbb I}
\newcommand{\N }{\mathbb N}
\newcommand{\Q }{\mathbb Q}
\newcommand{\q }{\mathsf Q}
\newcommand{\p }{\mathsf P}
\newcommand{\F }{\mathcal F}
\newcommand{\G }{\mathcal G}
\newcommand{\h }{\mathcal H}
\newcommand{\M }{\mathcal M}
\newcommand{\eps}{\varepsilon}
\newcommand{\argmax}{\operatornamewithlimits{argmax}}

\newcommand{\nc}{\newcommand}
\nc{\bg}{\begin} \nc{\e}{\end} \nc{\bi}{\begin{itemize}} \nc{\ei}{\end{itemize}} \nc{\be}{\begin{enumerate}} \nc{\ee}{\end{enumerate}} \nc{\bc}{\begin{center}} \nc{\ec}{\end{center}} \nc{\bq}{\begin{equation}} \nc{\eq}{\end{equation}} \nc{\la}{\label} \nc{\fn}{\footnote} \nc{\bs}{\backslash} \nc{\ul}{\underline} \nc{\ol}{\overline} \nc{\np}{\newpage} \nc{\btab}{\begin{tabular}} \nc{\etab}{\end{tabular}} \nc{\mc}{\multicolumn} \nc{\mr}{\multirow} \nc{\cdp}{\cleardoublepage} \nc{\fns}{\footnotesize} \nc{\scs}{\scriptsize} \nc{\RA}{\Rightarrow} \nc{\ra}{\rightarrow} \nc{\rc}{\renewcommand} \nc{\bfig}{\begin{figure}} \nc{\efig}{\end{figure}} \nc{\can}{\citeasnoun} \nc{\vp}{\vspace} \nc{\hp}{\hspace}\nc{\LRA}{\Leftrightarrow}\nc{\LA}{\Leftarrow}\nc{\rr}{\q}

\nc{\ch}{\chapter}
\nc{\s}{\section}
\nc{\subs}{\subsection}
\nc{\subss}{\subsubsection}

\newtheorem{thm}{Theorem}[section]
\newtheorem{cor}[thm]{Corollary}
\newtheorem{lem}[thm]{Lemma}

\theoremstyle{remark}
\newtheorem{rem}[thm]{Remark}

\theoremstyle{remark}
\newtheorem{ex}[thm]{Example}

\theoremstyle{definition}
\newtheorem{df}[thm]{Definition}

\newenvironment{rcases}{
  \left.\renewcommand*\lbrace.
  \begin{cases}}
{\end{cases}\right\rbrace}

\setlength\parindent{0pt}
\pagestyle{plain}

\maketitle

\s{Introduction}

The goal of this paper is to introduce a new kind of augmentation of filtrations which is suitable for a change of probability measure associated to a strict local martingale. While it is safe and very convenient to work under the usual conditions when doing a change of probability measure where the density process is a uniformly integrable martingale, one must be more careful if one takes a non-uniformly integrable martingale or  a strict local martingale as a "potential" Radon-Nikodym density process. \\
Indeed it was already noted by Bichteler (\cite{Bichteler}), and later in \cite{newkind}, that in order to extend a consistent family of probability measures from $\bigcup_{t\geq0}\F_t$ to $\F_\infty=\sigma\left(\bigcup_{t\geq0}\F_t\right)$, one has to impose certain topological requirements on the probability space and one must refrain from the usual assumptions. This is however rather unsatisfactory in general, since results from stochastic analysis like the existence of regular versions of martingales do require some augmentation of the filtration. The existence of such versions is for example of interest whenever one considers an uncountable number of stochastic processes as it is often the case in dynamic optimization problems. This led the authors of \cite{newkind} to introduce a new kind of augmentation of filtrations, the natural  augmentation, that is compatible with the construction of a probability measure on $\F_\infty$ whose density process is defined via a non-uniformly integrable martingale. \\
While a positive strict local martingale $(X_t)$, i.e.~a positive local martingale which is not a true martingale, cannot directly serve as a Radon-Nikodym density process, it is still possible to construct a new measure $\q$ on $\F_{\tau^X-}$ by extending the consistent family of measures $\q_n$ defined on $\F_{\tau_n^X}$ by 
	\[\q_n=X_{\tau_n^X}.\p,\quad \tau_n^X=\inf\{t\geq0:\ X_t>n\}\wedge n,\quad \tau^X=\lim_{n\ra\infty}\tau_n^X,
\]
if the filtration on the underlying probability space is the right-augmentation of a so called standard system. Standard systems were introduced in \cite{para} and first used in the above context in \cite{exit}. Since in this case the measure $\q$ is only uniquely defined on the sub-$\sigma$-algebra $\F_{\tau^X-}$ and is generally not absolutely continuous with respect to $\p$ on $\F_t$ for all $t\in\R$, we cannot use the natural augmentation from \cite{newkind}. While the problem in \cite{newkind} was the inclusion of null-sets from $\F_\infty$ in the initial filtration $\F_0$, the problem now becomes even more severe in that one can no longer include any null-events that happen after time $\tau^X$ in the initial filtration $\F_0$. This leads us to introduce a new kind of augmentation of filtrations along a sequence of stopping times that is on the one hand rich enough to make classical results from stochastic analysis hold true up to some stopping time and that on the other hand still allows for the construction of the new probability measure.\\ 
This note is organized as follows: in the next section we introduce a new kind of augmentation of filtrations along an increasing sequence of stopping times and we establish the existence of nice versions of stochastic processes  up to some stopping time under the new augmentation. In section \ref{newmeas} we briefly review the construction of the above mentioned probability measure associated to a positive (strict) local martingale, before we apply the augmentation results from section \ref{augment} in this setting.

\s{$(\tau_n)$-natural assumptions}\label{augment}

Let $(\Omega,\F,(\F_t)_{t\geq0},\p)$ be a filtered probability space. We will start with a definition before stating the augmentation theorem.

\begin{df}
Let $(\tau_n)_{n\in\N}$ be an increasing sequence of stopping times on $(\Omega,\F,(\F_t)_{t\geq0},\p)$.
\bi
\item A subset $A\in\Omega$ is called $(\tau_n)_{n\in\N}$-negligible with respect to $(\Omega,\F,(\F_t)_{t\geq0},\p)$, iff there exists a sequence $(B_n)_{n\in\N}$ of subsets of $\Omega$, such that for all $n\in\N$, $B_n\in\F_{\tau_n},\ \p(B_n)=0$, and $A\subset\bigcup_{n\in\N}B_n$.
\item We say that the filtered probability space $(\Omega,\F,(\F_t)_{t\geq0},\p)$ is $(\tau_n)$-complete, iff all the $(\tau_n)$-negligible sets of $\Omega$ are contained in $\F_0$. It satisfies the $(\tau_n)$-natural conditions, iff it is $(\tau_n)$-complete and the filtration $(\F_t)_{t\geq0}$ is right-continuous.
\ei
\end{df}

Note that in the case of $\tau_n=n$, the above definition as well as the next theorem reduces to the case of the natural augmentation studied in \cite{Bichteler} and \cite{newkind}, where all ${\F}^+_n$-negligible sets for all $n\in\N$ are included in ${\F}_0$. Thus, the following theorem can be seen as a generalization of Proposition 2.4 in \cite{newkind}.
 
\begin{thm} 
Let $(\tau_n)_{n\in\N}$ be an increasing sequence of stopping times on $(\Omega,\F,(\F_t)_{t\geq0},\p)$ and denote by $\mathcal{N}$ the family of all $(\tau_n)$-negligible sets with respect to $\p$. Set $\tilde{\F}=\sigma(\F,\mathcal{N})$ and $\tilde{\F}_{t}=\sigma(\F^+_t,\mathcal{N})$ for all $t\geq0$. Then there exists a unique probability measure $\tilde{\p}$ on $(\Omega,\tilde{\F})$, which coincides with $\p$ on $\F$, and the space  $(\Omega,\tilde{\F},(\tilde{\F}_t)_{t\geq0},\tilde{\p})$ satisfies the $(\tau_n)$-natural conditions. Moreover, $(\Omega,\tilde{\F},(\tilde{\F}_{t})_{t\geq0},\tilde{\p})$ is the smallest extension of $(\Omega,\F,({\F}_{t})_{t\geq0},{\p})$, which satisfies the $(\tau_n)$-natural conditions. We therefore call it the $(\tau_n)$-augmentation of $(\Omega,\F,({\F}_{t})_{t\geq0},{\p})$.
\end{thm}

\begin{proof}
We only give a sketch of the proof, because all steps except the third one (which we do in detail) follow closely the proof of Proposition 2.4.~in \cite{newkind}.
\be
\item
Define $\mathcal{E}=\{A\subset\Omega|\ \exists A'\in\F:\ A\Delta A'\in\mathcal{N}\}$. As in \cite{newkind} it is easily checked that $\mathcal{E}$ is a $\sigma$-algebra and that $\mathcal{E}=\tilde{\F}$. This implies that if $\tilde{\p}$ is a probability on $(\Omega,\tilde{\F})$ extending $\p$ we must have $\tilde{\p}(A)=\p(A')$ for $A\in\tilde{\F}$, where $A'\in\F$ satisfies $A\Delta A'\in\mathcal{N}$. Therefore, the measure $\tilde{\p}$ is unique, if it exists. Furthermore, $\tilde{\F}_t=\{A\subset\Omega|\ \exists A'\in\F^+_{t}:\ A\Delta A'\in\mathcal{N}\}$ for all $t\geq0$ as can be easily checked. 

\item
Next we show that $(\tilde{\F}_t)_{t\geq0}$ is right-continuous:\\
For this assume that $A\in\bigcap_{s>t}\tilde{\F}_s$. Therefore, $A\in\tilde{\F}_{t+1/n}$ for all $n\in\N$ and there exists $A'_n\in\F_{(t+1/n)+}$ such that $A\Delta A'_n\in\mathcal{N}$ for all $n\in\N$. Thus,
	\[A\Delta \underbrace{\left[ \bigcap_{m\geq m_0}\bigcup_{n\geq m}A_n' \right]}_{\in\F_{(t+1/m_0)+}} \in\mathcal{N}\ \forall\ m_0\in\N
\] 
and it follows that $A':=\bigcap_{m\in\N}\bigcup_{n\geq m}A_n'\in\F_{t+}$, which implies that $A\in\tilde{\F}_t$.

\item
The crucial step now is to show that $\sigma(\F^+_T,\mathcal{N})=\{A\subset\Omega|\ \exists A'\in\F^+_{T}:\ A\Delta A'\in\mathcal{N}\}$ for every $(\F_t^+)_{t\geq0}$-stopping time $T$: \\
Indeed it is well-known that $T$ can be approximated from above by a sequence of simple stopping times. Because of the right-continuity of the filtration, it is therefore enough to show the claim for every simple $(\F_t^+)$-stopping time $S$. For this assume that $S$ takes values in $\{t_1,\dots,t_k,\infty\}$ with $0\leq t_1<t_2<\dots<t_k<\infty$. Then we have
\begin{eqnarray*}
\tilde{\F}_S&=&\{A\in\tilde{\F}\ |\ A\cap\{S\leq t\}\in\tilde{\F}_t\ \forall\ t\geq0\}\\
&=&\{A\in\tilde{\F}\ |\ A\cap\{S\leq t_l\}\in\tilde{\F}_t\ \forall\ t\in [t_l,t_{l+1}) \ \forall\ l=1,\dots,k\}\\
&=&\{A\in\tilde{\F}\ |\ A\cap\{S\leq t_l\}\in\tilde{\F}_{t_l}\ \forall\ l=1,\dots,k\}\\
&=&\tilde{\F}_{S'}\cap\{A\in\tilde{\F}\ |\ A\cap\{S\leq t_1\}\in\tilde{\F}_{t_1}\}\\
&=&\{A\in\tilde{\F}_{S'}\ |\ A\cap\{S\leq t_1\}\in\tilde{\F}_{t_1}\},
\end{eqnarray*}
where $S'=S\vee t_2$. We will proceed by induction. Note that $S'$ takes only the values $\{t_2,\dots,t_k,\infty\}$ and by the induction hypothesis therefore $\tilde{\F}_{S'}=\sigma(\F_{S'}^+,\mathcal{N})$. 

Let $A\in\tilde{\F}_S$. Then $A\in\tilde{\F}_{S'}=\sigma(\F^+_{S'},\mathcal{N})$, which yields the existence of a set $A_0\in\F_{S'}^+$ such that $A\Delta A_0\in\mathcal{N}$ and $A_0\cap\{S'\leq t_l\}=A_0\cap\{S\leq t_l\}\in\F_{t_l}^+$ for all $l\in\{2,\dots,k\}$. Furthermore, since $A\cap\{S\leq t_1\}\in\tilde{\F}_{t_1}$, there exists some set $A_1\in\F^+_{t_1}$ such that \mbox{$A_1\Delta(A\cap\{S\leq t_1\})\in\mathcal{N}$}. 
Define $\ol{A}:=(A_0\cap\{S> t_1\})\cup (A_1\cap\{S\leq t_1\})$. Then $\ol{A}\in\F^+_S$:
\begin{eqnarray*}
\ol{A}\cap\{S\leq t_1\}&=&A_1\cap\{S\leq t_1\}\in\F_{t_1}^+,\\
\ol{A}\cap\{S\leq t_l\}&=& \underbrace{(A_1\cap\{S\leq t_1\})}_{\in\F^+_{t_1}} \cup\left(\underbrace{\{S>t_1\}}_{\in\F^+_{t_1}}\cap \underbrace{(A_0\cap\{S\leq t_l\})}_{\in\F_{t_l}^+}\right)\in\F^+_{t_l}\ \forall\ l=2,\dots,k.
\end{eqnarray*}
Moreover, 
\begin{eqnarray*}
\ol{A}\Delta A&=&\left[(A_0\cap\{S> t_1\})\cup (A_1\cap\{S\leq t_1\})\right]\Delta A\\ 
&=&\left(\left[(A_0\cap\{S> t_1\})\cup (A_1\cap\{S\leq t_1\})\right]\bs A\right) \cup \left(A\bs\left[(A_0\cap\{S> t_1\})\cup (A_1\cap\{S\leq t_1\})\right]\right)\\
&\subset&(A_0\bs A)\cup \left[A_1\bs (A\cap\{S\leq t_1\})\right]\cup \left[A\bs(A_0\cup A_1)\right]\cup \left[A\bs (A_1\cup\{S>t_1\})\right]\cup\left[A\bs (A_0\cup\{S\leq t_1\})\right]\\
&\subset&(A_0\bs A)\cup \left[A_1\bs (A\cap\{S\leq t_1\})\right]\cup \left[A\bs(A_0\cup A_1)\right]\cup \left[A\bs (A_1\cup\{S>t_1\})\right]\cup\left[A\bs (A_0\cup\{S\leq t_1\})\right]\\
&\subset&(A_0\bs A)\cup \left[A_1\bs (A\cap\{S\leq t_1\})\right]\cup \left[(A\cap\{S\leq t_1\})\bs A_1)\right]\cup (A\bs A_0)\\
&=& \left(A\Delta A_0\right) \cup \left[A_1\Delta(A\cap\{S\leq t_1\})\right]\in\mathcal{N}.
\end{eqnarray*}

Therefore, the claim follows by induction, once we show that it holds for the stopping time $S^*\in\{t_1,\infty\}$. For this note that
\begin{eqnarray*}
\tilde{\F}_{S^*}&=&\{A\in\tilde{\F}\ |\ A\cap\{S^*\leq t_1\}\in\tilde{\F}_{t_1}\}.
\end{eqnarray*}
Let $B\in\tilde{\F}_{S^*}$. Then there exists $B_1\in\F^+_{t_1}$ such that $B_1\subset\{S^*\leq t_1\}$ and $B_1\Delta(B\cap\{S^*\leq t_1\})\in\mathcal{N}$. Also, there exists $B_0\in\F$ such that $B\Delta B_0\in\mathcal{N}$. Now define $\ol{B}=B_1\cup(B_0\cap\{S^*>t_1\})\in\F$. Then $\ol{B}\cap\{S^*\leq t_1\}=B_1\in\F_{t_1}^+$ and 
\begin{eqnarray*}
\ol{B}\Delta B &=& (B_1\cup(B_0\cap\{S^*>t_1\}))\Delta B\\ 
&=& (B_1\cup(B_0\cap\{S^*>t_1\}))\Delta((B\cap\{S^*\leq t_1\}) \cup (B\cap\{S^*>t_1\}))\\
&\subset&\underbrace{(B_1\Delta(B\cap\{S^*\leq t_1\})}_{\in\mathcal{N}}\cup(\underbrace{(B_0\Delta B)}_{\in\mathcal{N}}\cap\{S^*>t_1\})\in\mathcal{N}.
\end{eqnarray*}
Therefore, $\ol{B}\in\F^+_{S^*}$ and $\tilde{\F}_{S^*}\subset\sigma(\F_{S^*}^+,\mathcal{N})$. The inclusion $\tilde{\F}_{S^*}\supset\sigma(\F_{S^*}^+,\mathcal{N})$ is trivial.\\
Finally, let $T$ be an arbitrary $(\F_t^+)_{t\geq0}$-stopping time and $(T_n)_{n\in\N}$ a decreasing sequence of simple stopping times such that $T_n\ra T$ from above. Then, since the filtration is right-continuous,
	\[\tilde{\F}_T=\bigcap_{n\in\N}\tilde{\F}_{T_n}=\bigcap_{n\in\N}\sigma(\F^+_{T_n},\mathcal{N})\supset\sigma(\F_T^+,\mathcal{N}).
\]
To show the reverse inclusion take any set $A\in\tilde{\F}_T$. From the above equality we get for each $n\in\N$ the existence of a set $A_n\in\F^+_{T_n}$ such that $A_n\Delta A\in\mathcal{N}$. We define $A'_n=\bigcup_{m\geq n}A_m\in\F_{T_n}^+$ for all $n\in\N$. Note that $(A'_n)_{n\in\N}$ is decreasing and that
	\[A_n'\Delta A =\left(\bigcup_{m\geq n}A_m\right)\Delta A=\left(\bigcup_{m\geq n}A_m\bs A\right)\cup\left( A\bs\bigcup_{m\geq n}A_m\right)=
\left(\bigcup_{m\geq n}\underbrace{A_m\bs A}_{\in\mathcal{N}}\right)\cup\left( \bigcap_{m\geq n}\underbrace{A\bs A_m}_{\in\mathcal{N}}\right) \in\mathcal{N}.
\]
By the right-continuity of the filtration, we have 
	\[A':=\bigcap_{n\in\N}\bigcup_{m\geq n}A_m=\bigcap_{n\in\N}A_n'\in\F_T^+.
\]
It remains to show that $A'\Delta A\in\mathcal{N}$. Indeed:
	\[A'\Delta A=\left(\bigcap_{n\in\N}A_n'\right)\Delta A=
	\left(\bigcap_{n\in\N}A_n'\bs A\right)\cup\left(A\bs\bigcap_{n\in\N}A_n'\right)=
	\left(\bigcap_{n\in\N}\underbrace{A_n'\bs A}_{\in\mathcal{N}}\right)\cup\left(\bigcup_{n\in\N}\underbrace{A\bs A_n'}_{\in\mathcal{N}}\right)\in\mathcal{N}.
\]
Therefore, $A\in\sigma(\F^+_T,\mathcal{N})$ and $\tilde{\F}_T\subset\sigma(\F^+_T,\mathcal{N})$.
 
\item
To show existence of the $(\tau_n)$-augmentation we define for $A\in\tilde{\F}$, $\tilde{\p}(A):=\p(A')$, where $A'\in\F$ satisfies $A\Delta A'\in\mathcal{N}$. This definition does not depend on the particular choice of $A'$. Obviously, $\tilde{\p}|_{\F}=\p$ and it is easily checked that $\tilde{\p}$ is $\sigma$-additive.
It remains to verify that $(\Omega,\tilde{\F},(\tilde{\F}_t)_{t\geq0},\tilde{\p})$ satisfies the $(\tau_n)$-natural conditions: If $A\in\tilde{\F}$ is $(\tau_n)$-negligible, then there exist $(B_n)_{n\in\N}$ such that $B_n\in\tilde{\F}_{\tau_n}$, $\tilde{\p}(B_n)=0$ for all $n\in\N$ and $A\subset\bigcup_{n\in\N}B_n$. Since $B_n\in\tilde{\F}_{\tau_n}$, there exists $B_n'\in\F^+_{\tau_n}$ such that $B_n\Delta B_n'\in\mathcal{N}$. Thus, $\p(B_n')=\tilde{\p}(B_n)=0$ and $B_n'\in\mathcal{N}$, which implies that also $B_n=(B_n'\cup(B_n\backslash B_n'))\backslash(B_n'\backslash B_n)\in\mathcal{N}$. It follows that $A\subset\bigcup_{n\in\N}B_n\in\mathcal{N}\subset\tilde{\F}_0$. 
Finally, it is easy to see that $(\Omega,\tilde{\F},(\tilde{\F}_t)_{t\geq0},\tilde{\p})$ is the smallest extension of $(\Omega,\F,(\F_t)_{t\geq0},\p)$ that satisfies the $(\tau_n)$-natural assumptions.
\ee
\end{proof}

\subs{Martingales under the $(\tau_n)$-natural augmentation}

We have the following simple but important result which shows that martingale properties of stochastic processes are not changed when taking the $(\tau_n)$-natural augmentation.

\begin{lem}\label{exp} {\bf (similar to Prop.~4.6~in \cite{newkind})}
Let $(\Omega,\F,(\F_t)_{t\geq0},\p)$ be a filtered probability space and $(\Omega,\tilde{\F},(\tilde{\F}_t)_{t\geq0},\tilde{\p})$ its $(\tau_n)$-augmentation with respect to an increasing sequence of $(\F_t)$-stopping times $(\tau_n)_{n\in\N}$. Let $X$ be an $\F$-measurable $\p$-integrable random variable. Then $X$ is also integrable with respect to $\tilde{\p}$ and $\E^{\tilde{\p}} X=\E^\p X$. Moreover, 
$\E^{\tilde{\p}}(X|\tilde{\F}_t)=\E^\p(X|\F_t)$ $\tilde{\p}$-a.s.~for all $t\geq0$.
\end{lem}

The proof is omitted, since it is exactly the same as the proof of Proposition 4.6 in \cite{newkind}.

\begin{cor}\label{local} Let $(\Omega,\F,(\F_t)_{t\geq0},\p)$ be a filtered probability space and $(\Omega,\tilde{\F},(\tilde{\F}_t)_{t\geq0},\tilde{\p})$ its $(\tau_n)$-augmentation, where $(\tau_n)_{n\in\N}$ is an increasing sequence of $(\F_t)$-stopping times.
\be
\item 
If $(X_t)_{t\geq 0}$ is an $\left\{(\F_t)_{t\geq0},\p\right\}$-(super-)martingale, then it is also an $\left\{(\tilde{\F}_t)_{t\geq0},\tilde{\p}\right\}$-(super)martingale.
\item If $(X_t)_{t\geq 0}$ is a local $\left\{(\F_t)_{t\geq0},\p\right\}$-martingale, then it is also a local $\left\{(\tilde{\F}_t)_{t\geq0},\tilde{\p}\right\}$-martingale.
\ee
\end{cor}

\begin{proof} Obviously, $(X_t)_{t\geq0}$ is $(\tilde{\F}_t)_{t\geq0}$-adapted and by Lemma \ref{exp} $X_t$ is integrable for all $t\geq0$.
\be
\item Furthermore, $\E^{\tilde{\p}}(X_t|\tilde{\F}_s)=\E^\p(X_t|\F_s)\stackrel{(>)}{=} X_s$ for all $s\leq t$ by Lemma \ref{exp}.
\item Let $(\sigma_n)_{n\in\N}$ be a localizing sequence for $X$ under $\p$. Since $(\sigma_n)_{n\in\N}$ are $(\F_t)_{t\geq0}$-stopping times, $\{\sigma_n\leq t\}\in\F_t\subset\tilde{\F}_t$ for all $t\geq0$ and $\tilde{\p}(\sigma_n\ra\infty)=\p(\sigma_n\ra\infty)=1$, thus $(\sigma_n)_{n\in\N}$ is also a localizing sequence for $X$ with respect to $(\Omega,\tilde{\F},(\tilde{\F}_t)_{t\geq0},\tilde{\p})$. By 1.~and Lemma \ref{exp} $(X_{t\wedge\sigma_n})_{t\geq0}$ is a uniformly integrable $\left\{(\tilde{\F}_t)_{t\geq0},\tilde{\p}\right\}$-martingale for all $n\in\N$.
\ee
\end{proof}

In the following subsection we show that one can do even better: in fact, it is possible to construct for any martingale an (adapted) version with regular trajectories for all $\omega\in\Omega$ up to time $\tau=\lim_{n\ra\infty}\tau_n$.

\subs{Existence of regular versions of trajectories up to time $\tau$}

As in \cite{newkind} the following lemma, which relates the $(\tau_n)$-natural conditions to the usual assumptions, is the main tool for establishing classical results from stochastic calculus under the $(\tau_n)$-natural conditions.

\begin{lem}\label{key}{\bf (similar to Prop.~2.5~in \cite{newkind})}
Assume that the filtered probability space $(\Omega,\F,(\F_t)_{t\geq0},\p)$ satisfies the $(\tau_n)$-natural assumptions for an increasing sequence of stopping times $(\tau_n)_{n\in\N}$. Then for all $n\in\N$ the space $(\Omega,\F_{\tau_n},(\F_{t\wedge\tau_n})_{t\geq0},\p)$ satisfies the usual assumptions.
\end{lem}

\begin{proof}
Let $A$ be an $\F_{\tau_n}$-negligible set, i.e.~there exists $B\in\F_{\tau_n}$ such that $A\subset B$ and $\p(B)=0$. Thus, $A$ is $(\tau_n)$-negligible with respect to $(\Omega,\F,(\F_t)_{t\geq0},\p)$, which is assumed to be $(\tau_n)$-complete. Therefore, $A\in\F_0$.
\end{proof}

For the rest of this subsection let $(\Omega,\F,(\F_t)_{t\geq0},\p)$ be a filtered probability space that satisfies the $(\tau_n)$-natural assumptions for an increasing sequence of stopping times $(\tau_n)_{n\in\N}$. Denote $\tau=\lim_{n\ra\infty}\tau_n$. Then on the subspace $(\Omega,\F_{\tau-},\p)$ many classical results from stochastic analysis can be proven to be true in a similar way as it is done in section 3 of \cite{newkind} under the natural assumptions. As an illustration of the usefulness of the $(\tau_n)$-usual assumptions we prove the existence of nice versions on $[0,\tau)$ below.

\begin{thm}{\bf (similar to Prop.~3.1 in \cite{newkind})}
Let $(X_t)_{t\geq0}$ be a supermartingale with respect to $(\Omega,\F,(\F_t)_{t\geq0},\p)$. If $t\mapsto\E^\p X_{t\wedge\tau_n}$ is right-continuous for all $n\in\N$, then $(X_t)_{0\leq t<\tau}$ admits a c\`adl\`ag modification on $(\Omega,\F_{\tau-},(\F_{t\wedge\tau-})_{t\geq0},\p)$, which is $\p$-a.s.~unique.
\end{thm}

\begin{proof}
Since the filtration is in particular assumed to be right-continuous, there exists a right-continuous adapted version of $(X_t)_{t\geq0}$ by Lemma (1.1) in \cite{exit}. Denote it by $\ol{X}$. Then by Doob's optional sampling theorem the process $(\ol{X}_{t\wedge\tau_n})_{t\geq0}$ is also a right-continuous supermartingale for every $n\in\N$, which is adapted to $(\F_{t\wedge\tau_n})_{t\geq0}$. But the space $(\Omega,\F_{\tau_n},(\F_{t\wedge\tau_n})_{t\geq0},\p)$ satisfies the usual conditions by Lemma \ref{key}. Thus, $(\ol{X}_{t\wedge\tau_n})_{t\geq0}$ admits a c\`adl\`ag modification, since $t\mapsto\E^\p(\ol{X}_{t\wedge\tau_n})=\E^\p(X_{t\wedge\tau_n})$ is right-continuous. Let us denote this modification by $\tilde{X}_t^{(n)}$, which is unique up to indistinguishability. Then $\tilde{X}_t^{(n)}$ is also a modification of $(X_{t\wedge\tau_n})_{t\geq0}$, and the uniqueness implies that the family $(\tilde{X}^{(n)})_{n\in\N}$ is consistent, i.e.~$\tilde{X}_t^{(n+k)}\indic_{\{t\leq \tau_n<\infty\}}=\tilde{X}_t^{(n)}\indic_{\{t\leq \tau_n<\infty\}}$ $\p$-almost surely for all $t\geq0$ and $n,k\in\N$. We define the set
	\[N:=\left\{\omega\in\Omega:\ \exists\ n,m\in\N,\ n\geq m,\ \exists\ t\in[0,\tau_m]\cap\R_+\text{ s.t.~}\tilde{X}_s^{(n)}(\omega)\neq\tilde{X}_t^{(m)}(\omega)\right\},
\]
which is $(\tau_n)$-negligible. Therefore, $N\in\F_0$ and $\p(N)=0$. Defining the process $(\tilde{X}_t)_{0\leq t<\tau}$ on $(\Omega,\F_{\tau-},\p)$ via 
	\[\tilde{X}_t(\omega)=\begin{cases}\tilde{X}_t^{(n)}(\omega)&\text{, if }\omega\not\in N \\ 0&\text{, if }\omega\in N\end{cases}
\]
for $t\in[0,\tau_n]$, we have constructed the desired c\`adl\`ag modification of $(X_t)_{0\leq t<\tau}$ on $(\Omega,\F_{\tau-},\p)$.
\end{proof}

\vskip12pt

\begin{thm}{\bf (similar to Prop.~3.3 in \cite{newkind})}
Let $(X_t)_{t\geq0}$ be an adapted process on the space $(\Omega,\F,(\F)_{t\geq0},\p)$ and assume that there exists a c\`adl\`ag version $(Y_t)_{0\leq t<\tau}$ of $(X_t)_{0\leq t<\tau}$, i.e.~for all $t\geq0$ s.t.~$\p(\tau>t)>0$ we have $\p(Y_t\neq X_t\ |\ t<\tau)=0$. Then there exists a c\`adl\`ag and adapted version of $(X_t)_{0\leq t<\tau}$ on $(\Omega,\F_{\tau-},(\F_{t\wedge\tau-})_{t\geq0},\p)$, which is indistinguishable from $(Y_t)_{0\leq t<\tau}$ on $(\Omega,\F_{\tau-},\p)$.
\end{thm}

\begin{proof}
We define the stopping times $(\tau_n^m)_{n,m\in\N}$ by
	\[\tau_n^m:=\sum_{k=1}^\infty\frac{k}{2^m}\indic_{\left\{\frac{k-1}{2^m}\leq \tau_n<\frac{k}{2^m}\right\}}.
\]
Then each $\tau_n^m$ takes only countably many values, $\tau_n^{m}\geq\tau_n^{m+1}\geq\tau_n$ and $\tau_n^m\ra\tau_n$ as $m\ra\infty$. Set
	\[D=\left\{\frac{k}{2^m}:\ k,m\in\N\right\},
\]
which is a countable dense subset of $\R_+$. Furthermore, define the function $f_\omega:D\ra\R$ via $f_\omega(t)=X_t(\omega)$. Then for all $n,m\in\N$ the set
	\[N_{n,m}=\left\{\omega\in\Omega:\ f_\omega|_{[0,\tau_n^m(\omega)]\cap D}\text{ does not admit a unique c\`adl\`ag extension to } [0,\tau_n^m(\omega)]\right\}
\]
is $\F_{\tau_n^m}$-measurable by Lemma 3.2 in \cite{newkind}, since $(X_{t\wedge\tau_n^m})_{t\geq0}$ is adapted to $(\F_t\cap\F_{\tau_n^m})_{t\geq0}$.
Furthermore, $N_{n,m}\supset N_{n,m+1}$ for all $n,m\in\N$ and
	\[N_n:=\bigcap_{m\in\N}N_{n,m}\in\F_{\tau_n},
\]
since the filtration is right-continuous.  Because $(Y_t)_{0\leq t<\tau}$ is a c\`adl\`ag version of $(X_t)_{0\leq t<\tau}$, we must have
	\[N_n\subset\{\omega\in\Omega\ |\ \exists\ t \in D\cap[0,\tau):\ X_t(\omega)\neq Y_t(\omega)\}=:C
\]
for all $n\in\N$. Since $(Y_t)_{0\leq t<\tau}$ is a version of $(X_t)_{0\leq t<\tau}$ on $(\Omega,\F_{\tau-},\p)$, $\p(C)=0$ and therefore also $\p(N_n)=0$ for all $n\in\N$, which implies that $N:=\bigcup_{n\in\N}N_{n}$ is $(\tau_n)$-negligible, i.e.~$\p(N)=0$ and $N\in\F_0$. Now, for $\omega\not\in N$ let $g_{\omega,n}$ be the unique c\`adl\`ag extension of the function $f_{\omega,n}:=f_{\omega}|_{[0,\tau_n)}$ from $D\cap[0,\tau_n)$ to $[0,\tau_n)$. By uniqueness the functions $(g_{\omega,n})_{n\in\N}$ are consistent, implying the existence of a c\`adl\`ag function $g_\omega:\ [0,\tau)\ra\R$ such that $g_\omega(t)=X_t(\omega)$ for all $t\in D\cap[0,\tau)$. Next, we define the c\`adl\`ag process $(\ol{X}_t)_{0\leq t<\tau}$ by $\ol{X}_t(\omega)=g_\omega(t)\indic_{\{\omega\not\in N\}}$. Indeed, for all $t^\omega<\tau(\omega)$ and for every sequence $(t^\omega_n)_{n\in\N}\subset D\cap [0,\tau(\omega))$ tending to $t^\omega$ from above, we have
	\[\ol{X}_{t^\omega}(\omega)=\indic_{\{\omega\not\in N\}}\lim_{n\ra\infty}g_\omega(t^\omega_n)=\indic_{\{\omega\not\in N\}}\lim_{n\ra\infty}X_{t^\omega_n}(\omega).
\]
Because the filtration is right-continuous and $N\in\F_0$, the adaptedness of $(X_t)_{t\geq0}$ implies that the process $(\ol{X}_t)_{0\leq t<\tau}$ is adapted to $(\F_{t\wedge\tau-})_{t\geq0}$. Since $Y$ is a version of $X$, $\ol{X}_t(\omega)$ is the right limit of $Y$ at $t$ restricted to $D\cap[0,\tau)$ for all $t\in[0,\tau)$ for almost all $\omega\in\Omega$. But since both, $Y$ and $\ol{X}$, are c\`adl\`ag on $(\Omega,\F_{\tau-},\p)$, $\ol{X}_t=Y_t$ for $t\in[0,\tau)$ $\p$-almost surely. Since $\p(Y_t\neq X_t\ |\ t<\tau)=0$ for all $t\geq0$ with $\p(\tau>t)>0$ by assumption, $\ol{X}$ is a c\`adl\`ag and adapted version of $X$ on $(\Omega,\F_{\tau-},(\F_{t\wedge\tau-})_{t\geq0},\p)$. Moreover, since both, $\ol{X}$ and $Y$, are c\`adl\`ag versions of $X$ on $(\Omega,\F_{\tau-},(\F_{t\wedge\tau-})_{t\geq0},\p)$, they must be indistinguishable. 
\end{proof}

It should be obvious that in a similar way other classical results of stochastic analysis like the Doob-Meyer decomposition or the existence of stochastic integrals can be proven up to time $\tau$. We will not go in any more details here, but instead concentrate on the application of the $(\tau_n)$-natural augmentation in the context of the extension of probability measures associated to strict local martingales in the next section. 

\s{Change of measure by a (strict) local martingale}\label{newmeas}

We briefly review the construction of a probability measure associated to a positive (strict) local martingale. For more details the reader may consult \cite{split1}.\\

In the following let $(\Omega,\F,(\F_t)_{t\geq0},\p)$ be a filtered probability space. Furthermore, we denote by $({\F}^+_t)_{t\geq0}$ the right-continuous augmentation of $(\F_t)_{t\geq0}$, i.e.~${\F}^+_t:=\F_{t+}=\bigcap_{s>t}\F_s$ for all $t\geq0$. Note that for now the filtration is \textit{not} completed with any negligible set of $\F$. In order to be able to construct the measure $\q$ associated with a (strict) local martingale $X$ mentioned in the introduction, the underlying probability space has to fulfill certain topological requirements.

\begin{df}
{\bf (cf.~\cite{exit})} Let $(\F_t)_{t\in T}$ be a filtration on $\Omega$, where $T$ is a partially ordered non-void index set. We call $(\F_t)_{t\in T}$ a standard system if
\bi
\item each measurable space $(\Omega,\F_t)$ is a standard Borel space, i.e.~$\F_t$ is $\sigma$-isomorphic to the $\sigma$-field of Borel sets on some complete separable metric space.
\item for any increasing sequence $(t_i)_{i\in\N}\subset T$ and for any $A_1\supset A_2\supset\dots\supset A_i\supset\dots$, where $A_i$ is an atom of ${\F}_{t_i}$, we have $\bigcap_i A_i\neq\emptyset$.
\ei
\end{df}

The most important examples of standard systems are the filtrations generated by the coordinate process on the spaces $C'(\R_+,\ol{\R}_+)$ or $D'(\R_+,\ol{\R}_+)$ of all non-negative continuous resp.~c\`adl\`ag functions $(\omega(t))_{t\geq0}$ that have left limits on $(0,\alpha(\omega))$ for some $\alpha(\omega)\in[0,\infty]$ and remain constant on $[\alpha(\omega),\infty)$ at the value $\lim_{t\uparrow\alpha(\omega)}\omega(t)$ if this limit exists and at $\infty$ otherwise. Note that the spaces $C(\R_+,\R)$ or $C([0,1],\R)$, endowed with   the filtrations generated by the coordinate process, are not standard systems. Adding the point $\{\infty\}$ is crucial.\\

\textit{Notation:}
When working on the subspace $(\Omega,\F_{\tau-})$ of $(\Omega,\F)$, where $\tau$ is some $(\F_t)$-stopping time, we must restrict the filtration to $(\F_{t\wedge\tau-})_{t\geq0}$, where with a slight abuse of notation we set $\F_{t\wedge\tau-}:=\F_t\cap\F_{\tau-}$.
In the following we may also write $(\F_t)_{0\leq t<\tau}$ for the filtration on $(\Omega,\F_{\tau-},\p)$.\\%, where we understand $[0,\tau)$ as the stochastic interval $[0,\tau)=\{(\omega,t):\ 0\leq t<\tau(\omega)\}$.\\

The following theorem is a generalization of Theorem 4 in \cite{DSbessel} and Proposition 1 in \cite{PalP} that deal with continuous local martingales on path spaces. Its proof relies on the construction of the F\"ollmer measure, cf.~\cite{exit} and can be found in \cite{split1}. In Theorem \ref{kompaug} below we will state a further extension of this result involving the new kind of augmentation of filtrations introduced in section \ref{augment}.

\begin{thm}\label{komplett}
Let $\left(\Omega,\F,(\F_t)_{t\geq 0},\p\right)$ be a filtered probability space and assume that $(\F_t)_{t\geq0}$ is a standard system. Let $X$ be a c\`adl\`ag local martingale on the space $(\Omega,\F,({\F}^+_{t})_{t\geq 0},\p)$ with values in $(0,\infty)$ and $X_0=1$ $\p$-almost surely. We define $\tau^X_n:=\inf\{t\geq 0:\ X_t>n\}\wedge n$ and $\tau^X=\lim_{n\ra\infty}\tau^X_n$. Then there exists a unique probability measure ${\q}$ on 
$\left(\Omega,{\F}^+_{\tau^X-},({\F}^+_{t\wedge\tau^X-})_{t\geq0}\right)$, such that $\frac{1}{X}$ is a ${\q}$-martingale up to time $\tau^X$. Furthermore, ${\q}|_{{\F}^+_{t}\cap\F^+_{\tau^X-}}\gg{\p}|_{{\F}^+_t\cap\F^+_{\tau^X-}}$ for all $t\geq0$ with Radon-Nikodym derivative given by $\left.\frac{d{\p}}{d{\q}}\right|_{{\F}^+_t\cap{\F}^+_{\tau^X-}}=\frac{1}{X_t}\indic_{\{t<\tau^X\}}=\frac{1}{X_t}$.\\
Moreover, $X$ is a strict local $\p$-martingale if and only if $\q(\tau^X<\infty)>0$.
\end{thm}

From here it is easy to see why we cannot work with the natural augmentation of $(\F_t)_{t\geq0}$, but will have to use the $(\tau_n^X)$-natural augmentation introduced in section \ref{augment}. Indeed, we have $A_t:=\{t\geq\tau^X\}\in\F^+_t\cap\F_{\tau^X-}$ and ${\p}(A_t)=0$ for all $t\geq0$, while 
	\[\q(A_t)=1-\q(\tau^X>t)=	1-\E^\p(X_t)>0
\]
for some $t>0$, if $X$ is a strict local martingale. Now, if $(\Omega,\F,(\F_t)_{t\geq0},\p)$ satisfied the natural conditions, then $A_t\in\F_0$ for all $t\geq0$ and since $\p|_{\F_0}=\q|_{\F_0}$ this would imply that $\q(A_t)=\p(A_t)=0$ for all $t\geq0$, an obvious contradiction.
\\

With the help of section \ref{augment} we can nevertheless state the following extension of Theorem \ref{komplett}:

\begin{thm}\label{kompaug}
Let $\left(\Omega,\F,(\F_t)_{t\geq 0},\p\right)$ be a filtered probability space and assume that $(\F_t)$ is a standard system. Let $X$ be a c\`adl\`ag local martingale on the space $(\Omega,\F,({\F}^+_t)_{t\geq 0},\p)$ with values in $(0,\infty)$ and $X_0=1$ $\p$-almost surely. We define $\tau^X_n:=\inf\{t\geq 0:\ X_t>n\}\wedge n$, $\tau^X:=\lim_{n\ra\infty}\tau^X_n$ and denote by $(\Omega,\tilde{\F},(\tilde{\F}_t)_{t\geq 0},\tilde{\p})$ the $(\tau^X_n)$-augmentation of $(\Omega,\F,({\F}_{t})_{t\geq 0},\p)$. Then there exists a unique probability measure $\tilde{\q}$ on $(\Omega,\tilde{\F}_{\tau^X-},(\tilde{\F}_{t\wedge\tau^X-})_{t\geq0})$, such that $\frac{1}{X}$ is a $\tilde{\q}$-martingale. Furthermore, $\tilde{\q}|_{\tilde{\F}_{t}}\gg\tilde{\p}|_{\tilde{\F}_t}$ for all $t\in[0,\tau^X)$ with Radon-Nikodym derivative  $\left.\frac{d\tilde{\p}}{d\tilde{\q}}\right|_{\tilde{\F}_{t}\cap\tilde{\F}_{\tau^X-}}=\frac{1}{X_t}\indic_{\{t<\tau^X\}}$.
\end{thm}

\begin{proof}
Let $\left(\Omega,\ol{\F}_{\tau^X-},(\ol{\F}_{t\wedge\tau^X-})_{t\geq0},\ol{\q}\right)$ be the $(\tau^X_n)$-augmentation of $\left(\Omega,\F^+_{\tau^X-},(\F^+_{t\wedge\tau^X-})_{t\geq0},\q\right)$ as constructed in Theorem \ref{komplett}. Then $\ol{\F}_{t\wedge\tau^X-}=\tilde{\F}_{t\wedge\tau^X-}$ for $t\geq0$ and $\tilde{\F}_{\tau^X-}=\ol{\F}_{\tau^X-}$: $A$ is $(\tau^X_n)$-negligible with respect to $\left(\Omega,\F,(\F^+_t)_{t\geq0},\p\right)$ iff there exist $(B_n)_{n\in\N}$ such that $A\subset\bigcup_{n\in\N}B_n$ and $B_n\in\F^+_{\tau^X_n},\ \p(B_n)=0$ for all $n\in\N$. Since $\q|_{\F^+_{\tau^X_n}}\sim\p|_{\F^+_{\tau^X_n}}$, $\q(B_n)=\p(B_n)=0$. Thus, $A$ is $(\tau^X_n)$-negligible with respect to $(\Omega,\F,(\F^+_t)_{t\geq0},\p)$ iff $A$ is $(\tau^X_n)$-negligible with respect to $\left(\Omega,\F^+_{\tau^X-},(\F^+_{t\wedge\tau^X-})_{t\geq0},\q\right)$, i.e.~$A\in\ol{\F}_{t\wedge\tau^X-}$ for $t\geq0$.\\ Now let $A\in\tilde{\F}_t$ for some $t\geq0$, i.e.~there exists $A'\in\F_t^+$ such that $A\Delta A'$ is $(\tau^X_n)$-negligible with respect to $\q$ and $\p$. Then:
	\[\tilde{\p}(A)=\tilde{\p}(A')=\p(A')=\E^\q\left(\indic_{\{A',\tau^X>t\}}\frac{1}{X_t}\right)=\E^{\ol{\q}}\left(\indic_{\{A',\tau^X>t\}}\frac{1}{X_t}\right)
	=\E^{\ol{\q}}\left(\indic_{\{A,\tau^X>t\}}\frac{1}{X_t}\right),
\]
i.e.~$\left.\frac{d\tilde{\p}}{d\ol{\q}}\right|_{\tilde{\F_t}\cap\tilde{\F}_{\tau^X-}}=\frac{1}{X_t}\indic_{\{\tau^X>t\}}$. Identifying $\tilde{\q}$ with $\ol{\q}$ yields the result.
\end{proof}

Let us briefly explain why the $(\tau_n^X)$-natural augmentation is "good enough" for the setup considered here. First note that the measure $\q$ is unevitably connected with the local martingale $X$. Therefore, it is not surprising that also the augmentation depends on the process $X$ itself. On the other hand every process $Y$ defined on $(\Omega,\F,(\F_t)_{t\geq0},\p)$ is only defined up to time $\tau^X$ under $\q$. Since one is normally interested in the $\p$-probability of events and uses the measure $\q$ just as a helpful device to infer something about the $\p$-probabilites, it is therefore almost always sufficient in applications to have results from stochastic analysis holding only until time $\tau^X$, because everything that happens with positive probability under $\p$ takes place before time $\tau^X$ $\q$-almost surely. \\

%One can however extend the measure $\q$ defined on $\F_{\tau^X-}$ to $\F_\infty=\bigvee_{t\geq0}\F_t$ by Theorem 3.1 in \cite{Ershov}, since $(\Omega,\F_\infty)$ is a standard Borel space and $\F_{\tau^X_}$ is countably generated. The most convenient way to define $Y_t$ for all $t\geq0$ is then to set $Y_t=\liminf_{s\ra\tau^X,s<\tau^X,s\in\Q}Y_s$ on $\{\tau^X\leq t\}$. Hence $Y$ has trivially regular paths from time $\tau^X$ on. The only remaining question is then whether $Y_{\tau^X-}$ exists in $\R_+$, which will not be the case in general, since it is related to the asymptotic behaviour of the process $Y$ under $\p$.

Last but not least let us point out two situations in which it seems important to have nice versions of trajectories, i.e.~processes which are regular everywhere and not only up to a nullset. Clearly, this is necessary if one considers an infinite number of stochastic processes. As already mentioned in the introduction this happens regularly in optimization problems as for example in portfolio optimization. Indeed, even if the number of stocks is finite, the set of admissible trading strategies is in general so rich that the set of possible portfolio value processes is uncountable.\\

As a second example consider the occupation times formula, which requires in its proof a jointly continuous version of the field of local times at all levels and points in time. However, it was shown in \cite{newkind} that without augmentation there does generally not exist a c\`adl\`ag and adapted version of the local time process at level $a\in\R$, i.e.~local times can explode in finite time on some set. Hence, if one wants to apply powerful results from stochastic analysis like the occupation times formula, one should work on augmented probability spaces.

%============================================================

%\np


\begin{thebibliography}{99}

\bibitem{Bichteler}
K.~Bichteler.
\newblock {\em {Stochastic integration with jumps}}.
\newblock {Encyclopedia of Mathematics and Its Applications. 89. Cambridge:
  Cambridge University Press}, 2002.

\bibitem{DSbessel}
F.~Delbaen and W.~Schachermayer.
\newblock {Arbitrage possibilities in Bessel processes and their relations to
  local martingales}.
\newblock {\em Probab. Theory Relat. Fields}, 102(3):357--366, 1995.

\bibitem{exit}
H.~F\"ollmer.
\newblock The exit measure of a supermartingale.
\newblock {\em Zeitschrift f\"ur Wahrscheinlichkeitstheorie und verwandte
  Gebiete}, 21:154--166, 1972.

\bibitem{split1}
C.~Kardaras, D.~Kreher, and A.~Nikeghbali.
\newblock Strict local martingales and bubbles.
\newblock Preprint, 2013.

\bibitem{newkind}
J.~Najnudel and A.~Nikeghbali.
\newblock A new kind of augmentation of filtrations.
\newblock {\em ESAIM: Probability and Statistics}, 15:S39--S57, 2011.

\bibitem{PalP}
S.~Pal and P.~Protter.
\newblock {Analysis of continuous strict local martingales via $h$-transforms}.
\newblock {\em Stochastic Processes Appl.}, 120(8):1424--1443, 2010.

\bibitem{para}
K.~Parthasarathy.
\newblock {\em Probability Measures on Metric Spaces}.
\newblock New York, London: Academic Press, 1967.

\end{thebibliography}
\end{document}